\documentclass[a4paper]{article}

\newcommand{\f}{\frac}

\newcommand{\R}{\mathbb R}
\newcommand{\C}{\mathbb C}
\newcommand{\N}{\mathbb N}

\newcommand{\eps}{\varepsilon}
\renewcommand{\epsilon}{\varepsilon}

\newcommand{\Om}{\Omega}

\newcommand{\dist}{\operatorname{dist}}

\newcommand{\V}{\mathcal{V}}
\newcommand{\dom}{\operatorname{dom}}
\newcommand{\h}{\mathcal{H}}
\newcommand{\id}{\mathrm{id}}

\newenvironment{enumi}{\begin{enumerate}[(i)]}{\end{enumerate}}

\AtBeginDocument{
	\def\l{\mathopen}
	\def\r{\mathclose}
}

\usepackage{etex}
\usepackage[english]{babel}
\usepackage{tikz}
\usetikzlibrary{arrows,decorations.pathmorphing,backgrounds,positioning,fit,petri,patterns}
\usepackage{pgfplots}
\usepackage[utf8]{inputenc}
\usepackage{graphicx}
\usepackage[colorlinks=true,linkcolor=black,citecolor=black]{hyperref}
\usepackage{amsfonts}
\usepackage{enumerate}
\usepackage{booktabs}
\usepackage{array} 
\usepackage{paralist} 
\usepackage{subfig} 
\usepackage{mathtools}
\usepackage{tabu}
\usepackage{amsthm}
\usepackage{empheq}
\usepackage{amsopn}
\usepackage{dsfont}
\usepackage{wrapfig}
\usepackage[font=small]{caption}
\usepackage{units}
\usepackage[symbol]{footmisc}
\usepackage{todonotes}
\usepackage{amssymb}
\usepackage{pdfpages}
\usepackage{setspace}

\usepackage{kvsetkeys}
\usepackage{etexcmds}

\makeatletter
\newcount\arg@count
\newcommand{\arg@parser}[1]{%
  \advance\arg@count\@ne
  \expandafter\let\csname arg\romannumeral\arg@count\endcsname\comma@entry
}
\newcommand\res[1]{
  \arg@count=\z@
  \comma@parse{ \lambda,A }\arg@parser 
  \arg@count=\z@
  \comma@parse{#1}\arg@parser
  \ifnum\arg@count>2 %
    \@latex@error{Too many arguments}{%
      The macro \string\mycmd\space got \the\arg@count\space
       arguments,\MessageBreak
      but expected are 2 arguments.\MessageBreak
      \@ehd
    }%
  \fi
  \edef\process@me{%
    \noexpand\@res
    {\etex@unexpanded\expandafter{\argi}}%
    {\etex@unexpanded\expandafter{\argii}}%
  }%
  \process@me
}
\newcommand{\@res}[2]{%
  \ensuremath\left( #1 - #2 \right)^{-1}
}
\makeatother

\makeatletter
\newcount\arg@count
\newcommand\p[1]{
  \arg@count=\z@
  \comma@parse{  }\arg@parser 
  \arg@count=\z@
  \comma@parse{#1}\arg@parser
  \ifnum\arg@count=2 %
  \else
    \@latex@error{Wrong number of mandatory arguments}{%
      The macro \string\p\space got \the\numexpr\arg@count-2\relax\space
      mandatory arguments,\MessageBreak
      but expected are 3 mandatory arguments.\MessageBreak
      \@ehd
    }%
  \fi
  \edef\process@me{%
    \noexpand\@p
    {\etex@unexpanded\expandafter{\argi}}%
    {\etex@unexpanded\expandafter{\argii}}%
  }%
  \process@me
}
\newcommand{\@p}[2]{%
  \ensuremath \left\langle #1 , #2 \right\rangle
}
\makeatother

\allowdisplaybreaks

\theoremstyle{definition}
\newtheorem{de}{Definition}

\theoremstyle{plain}
\newtheorem{prop}[de]{Proposition}
\newtheorem{lemma}[de]{Lemma}
\newtheorem{theorem}[de]{Theorem}

\newtheorem{corollary}[de]{Corollary}

\theoremstyle{remark}
\newtheorem{remark}[de]{Remark}

\providecommand{\keywords}[1]{\textbf{\textit{Keywords: }} #1}

\author{F. R\"osler\thanks{School of Mathematics,\newline 
Cardiff University \newline
Email: \href{mailto:roslerf@cardiff.ac.uk}{RoslerF@Cardiff.ac.uk}}}
\date{}
\title{A Note on Spectral Convergence in Varying Hilbert Spaces}

\begin{document}

\maketitle

\begin{abstract}
	We prove sufficient conditions for Hausdorff convergence of the spectra of sequences of closed operators defined on varying Hilbert spaces and converging in norm-resolvent sense, i.e. $\|J_\eps(1+A_\eps)^{-1} - (1+A)^{-1}J_\eps\|\to 0$ as $\eps\to 0$, where $J_\eps$ is a suitable identification operator between the domains of the operators. This is an extension of results of \cite{MNP}, who proved absence  of spectral pollution for sectorial operators.
\end{abstract}


\section{Introduction}

Convergence of spectra of sequences of operators has long been a subject of intense interest in asymptotic analysis. Many techniques, such as homogenisation, or dimensional reduction rely on convergence theorems for the spectra of sequences of operators. 

Classical theorems from perturbation theory (cf \cite{Kato}) give partial answers to the question of spectral convergence under different assumptions on the sequence. One central classical result is that under norm-resolvent convergence, so-called \emph{spectral pollution} is not possible, i.e. if $A_n$ converges to $A$ in norm resolvent sense, then the ``limit'' of the spectra of $A_n$ cannot be larger than the spectrum of $A$. More precisely, one has
\begin{theorem}[{\cite{Kato,RS1}}]\label{th:classical}
	If $A_n,A$ are closed operators on a Hilbert space $\h$ and there exists $\mu\in\rho(A)$ such that for $n$ large enough one has $\mu\in\rho(A_n)$ and $\|(\mu-A_n)^{-1}-(\mu-A)^{-1}\|\to 0$, as $n\to\infty$, then for any $\lambda\in\rho(A)$ there exists $n_0\in\N$ such that $\lambda\in\rho(A_n)$ for all $n>n_0$.
\end{theorem}
However, the converse of Theorem \ref{th:classical} is \emph{not} true, meaning that there exist sequences of operators for which the spectrum of the spectrum of the limit operator $A$ is much larger than the spectrum of any $A_n$ for finite $n$. Indeed, let us demonstrate this with an example (cf. \cite[Ch. IV.3.1]{Kato}).

\paragraph{Example.}
Let $\h=l^2(\mathbb{Z})$ and let $\{e_n\}$ be its canonical basis. For $n\in\N$ define $T_n\in L(\h)$ by
\begin{align*}
	T_ne_0 &:= n^{-1} e_{-1}\\
	T_ne_i &:= e_{i-1},\qquad i\neq 0.
\end{align*}
A straightforward calculation shows that $T_n-\lambda$ is boundedly invertible for every $\lambda\in\C$ with $|\lambda|<1$. Since also clearly $\|T_n\|\leq 1$ for all $n$, we conclude that $\sigma(T_n)\subset S^1$, the unit circle in $\C$.

Now consider the limit of $(T_n)$. By definition, $T_n$ is a rank-one perturbation of the operator $T$ defined by
\begin{align*}
	T_ne_0 &:= 0\\
	T_ne_i &:= e_{i-1},\qquad i\neq 0.
\end{align*}
It follows that $\|T_n-T\|_{L(\h)} = \f1n$ and hence $T_n$ converges to $T$ in operator norm. However, the spectrum of $T$ is considerably larger than the unit circle $S^1$. Indeed, another straightforward calculation shows that for every $\lambda\in\C$ with $|\lambda|<1$ the vector $x:=\sum_{n=0}^\infty \lambda^ne_n$ is an eigenvector of $T$ and thus $\sigma(T)$ contains the closed unit disk.

\

The above example shows that even under operator norm convergence, \emph{spectral inclusion} can fail, i.e. there may exist points $\lambda\in\sigma(A)$ such that there exists no sequence $\lambda_n\in\sigma(A_n)$ with $\lambda_n\to\lambda$.

There exist other examples demonstrating this lower semi-discontinuity of the spectrum (cf. \cite[p. 282]{Rickart}, \cite[Ch.1, §5]{aupetit} for an example in which the spectrum collapses from a disk to a point). Therefore, further assumptions are necessary in order to obtain spectral inclusion.

In the next section, we will set the scene and present our main results. The following sections contain the proofs of these results.

\section{Main Results}

In order to aim for generality, we consider families of operators $A_\eps$ which are not necessarily defined on the same Hilbert space. More precisely, for $\eps>0$ let $\h,\h_\eps$ be Hilbert spaces and let $A_\eps:\h_\eps\supset\dom(A_\eps)\to\h_\eps$ and $A:\h\supset\dom(A)\to\h$ be closed operators.

Let us denote $\V_\eps:=\big(\dom(A_\eps),\|\cdot\|_{A_\eps}\big)$ and $\V:=\big(\dom(A),\|\cdot\|_{A}\big)$, where $\|\cdot\|_A$ denotes the  graph norm of $A$, that is, $\|u\|_{\V}^2 := \|u\|_A := \|u\|_{\h}+\|Au\|_{\h}$ (analogously for $\|\cdot\|_{\V_\eps}$). 
\begin{de}\label{def:extended_norm-resolvent}
	Assume that there exists $z_0\in\bigcap_{\eps>0}\rho(A_\eps)\cap\rho(A)$ and operators
		$J_\eps: \h_\eps\to \h$ and $ I_\eps:\h\to\h_\eps$
	such that 
	\begin{enumi}
		\item $\|I_\eps J_\eps - \id_{\h_\eps}\|_{L(\V_\eps,\h_\eps)}\to 0$ as $\eps\to 0$,\label{(i)old}
		\item $J_\eps I_\eps \to \id_{\h}$ strongly as $\eps\to 0$,\label{(ii)new}
		\item $\|I_\eps\|_{ L(\h,\h_\eps)},\|J_\eps\|_{ L(\h_\eps,\h)}\leq M$ for some $M>0$ uniformly in $\eps$,\label{(ii)old}
		\item $\bigl\|J_\eps(z_0\,\id_{\h_\eps}-A_\eps)^{-1} -(z_0\,\id_{\h}-A)^{-1}J_\eps\bigr\|_{ L(\h_\eps,\h)}\to 0$ as $\eps\to 0$.\label{(iii)old}
	\end{enumi}
	Then we say that the sequence $(A_\eps)$ converges to $A$ in the \emph{extended norm resolvent sense}. 
\end{de}
\begin{remark}
	Note the asymmetry between the assumptions (i) and (ii) above, where we require convergence with respect to the operator norm $\|\cdot \|_{L(\V_\eps,\h_\eps)}$ and only \emph{strong} convergence for $J_\eps I_\eps$, allowing a great deal of freedom for the construction of $I_\eps,J_\eps$ in applications.

Moreover, note that if $\h_\eps\equiv \h$ for all $\eps>0$ and $I_\eps=J_\eps=\id_{\h}$ for all $\eps>0$, this definition reduces to the classical definition of norm resolvent convergence.
\end{remark}

The assumptions on the identification operators $I_\eps,J_\eps$ cover many cases relevant in applications. Examples include
\begin{enumi}
	\item \emph{Projection onto subspaces.} Let $\h_n$ be an increasing sequence of closed subspaces such that the orthogonal projection $P_n:\h\to\h_n$ converges strongly to the identity. Then defining $I_n:=P_n$ and $J_n:\h_n\hookrightarrow\h,\, J_n(x)=x$ satisfy assumptions (i)-(iii) of Definition \ref{def:extended_norm-resolvent}. Indeed, it is easy to check that $I_n J_n=\id_{\h_n}$, while the strong convergence of $J_\eps I_\eps$ follows from the strong convergence $P_n\to\id_\h$.
	\item \emph{Perforated domains.} Let $\Om\subset\R^d$ be open and $T_\eps\subset\Om$ be a collection of closed subsets such that $|T_\eps|\to 0$ as $\eps\to 0$, where $|\cdot|$ denotes the Lebesgue measure. Let $\h:=L^2(\Om), \h_\eps:=L^2(\Om\setminus T_\eps)$ and $\V:=H^1(\Om)$. Define $I_\eps,J_\eps$ by
	\begin{align*}
		J_\eps &: \h_\eps \to \h,
		&J_\eps f(x) &= \begin{cases}
			f(x), & x\in\Om\setminus T_\eps, \\[0.1em]
			0, & x\in T_\eps
		\end{cases}
		\\[1mm]
		I_\eps &: \h \to \h_\eps,
		&I_\eps g(x) &= g|_{\Om\setminus T_\eps}
	\end{align*}
	In this case we have again that $I_\eps J_\eps = \mathrm{id}_{\h_\eps}$ and $
		\|J_\eps I_\eps - \mathrm{id}_{\h}\|_{ L(\V,\h)} \to 0$. Indeed, the first equality is trivial, while the second follows by the following calculation. Let $f\in \V$. Then we have $\|f-J_\eps I_\eps f\|_{L^2(\Om)} = \|f\|_{L^2(T_\eps)}$. To show that this quantity converges to 0 uniformly in $f$, note that
	\begin{align*}
		\|f\|_{L^2(T_\eps)}^2 &\leq \left\|1\right\|_{L^{2p}( T_\eps)}^2 \left\|f\right\|_{L^{2q}(T_\eps)}^2 
	\end{align*}
	for $p,q>1$ with $p^{-1}+q^{-1}=1$, by H\"older's inequality. Since $f\in H^1(\Om)$, we can use the Gagliardo-Sobolev-Nierenberg inequality to conclude (for suitable $q$) that
	\begin{align*}
		\|f\|_{L^2(T_\eps)}^2 &\leq  \left\|1\right\|_{L^{2p}(T_\eps)}^2   C \left\|f\right\|_{H^1(\Om)}^2\\
		&= C |T_\eps|^{\nicefrac 1 p}  \left\|f\right\|_{H^1(\Om)}^2
	\end{align*}
	with some suitable $p>0$. Since $|T_\eps|\to 0$ as $\eps\to 0$, the desired convergence follows. By density of $\V$ in $\h$ we readily conclude the strong convergence $J_\eps I_\eps\to\id_\h$.
	
	Indeed, the main result of this paper complements the proof of spectral convergence in perforated domains in \cite{CDR}, which was only sketched there.
	\item \emph{Dimensional reduction.} Consider a domain $\Om_\eps\subset\R^{d+1}$ of the form $\Om_\eps = (\eps U)\times(0,1)$, where $U\subset\R^d$ is open and bounded. Define $\h_\eps:=L^2(\Om_\eps),\;\h:=L^2((0,1))$ and $\V_\eps:=H^1(\Om_\eps)$. For $f\in\h$, define $(I_\eps f)(x,t):=f(t)$ and for $u\in\h_\eps$ define $(J_\eps u)(t):=|\eps U|^{-d}\int_{\eps U}u(x,t)\,dx$, which is well-defined for almost every $t\in(0,1)$ by Fubini's theorem. This time, it is easily checked that $J_\eps I_\eps=\id_\h$. Moreover, one has \begin{align*}
			(u-I_\eps J_\eps u)(x,t) &= u(x,t) - |\eps U|^{-d}\int_{\eps U}u(y,t)\,dy\\
			&= |\eps U|^{-d}\int_{\eps U}(u(x,t) - u(y,t))\,dy\\
		\end{align*}
		and hence
		\begin{align*}
			\|(1-I_\eps J_\eps)u\|_{L^2(\Om_\eps)}^2 &=\int_{\eps U}\int_0^1\left||\eps U|^{-d}\int_{\eps U}(u(x,t) - u(y,t))\,dy\right|^2 dx\,dt\\
			&\leq \int_{\eps U}\int_0^1|\eps U|^{-d}\int_{\eps U}|u(x,t) - u(y,t)|^2\,dy\,dx\,dt\\
			&= |\eps\operatorname{diam}(U)|^2\int_{\eps U}\int_0^1|\eps U|^{-d}
			 \left\|\nabla u(\cdot,t)\right\|^2_{L^2(\eps U)}\,dx\,dt\\
			 &= \eps^2\operatorname{diam}(U)^2
			 \left\|\nabla u\right\|^2_{L^2(\Om_\eps)},
		\end{align*}
		where we have used Jensen's inequality. The above inequality shows that we have $\|\id_{\h_\eps}-I_\eps J_\eps\|_{L(\V_\eps,\h_\eps)}\leq C\eps$.
\end{enumi}

The main result of this article is the following.

\begin{theorem}\label{th:mainth}
	Let $A_\eps,A$ be closed operators on $\h_\eps,\h$ respectively, and assume that $A_\eps$ converges to $A$ in extended norm resolvent sense. Then one has
	\begin{enumi}
		\item If $\rho(A)$ is connected, then for every compact $K\subset\rho(A)$ there exists $\eps_0>0$ such that $K\subset\rho(A_\eps)$ for all $\eps\in(0,\eps_0)$.
		\item Assume that $K\subset\C$ is compact, connected such that $K\subset\rho(A_{\eps})$ and $\|(z-A_\eps)^{-1}\|_{L(\h_\eps)}\leq C$ for all $\eps>0,\,z\in K$. 
			Assume further that $K$ can be connected to $\{z_0\}$ by a curve $\gamma$ lying in $\rho(A_{\eps})$ for all $\eps>0$.  Then one has $K\subset\rho(A)$.
		\item For every isolated point $\lambda\in\sigma(A)$ such that $B_\delta(\lambda)\setminus\{\lambda\}$ is in the same connected component of $\rho(A)$ as $z_0$ there exists a sequence $\lambda_\eps\in\sigma(A_\eps)$ such that $\lambda_\eps\to\lambda$.
	\end{enumi}
\end{theorem}

The first part of the above theorem shows the absence of spectral pollution, while the second part shows spectral inclusion under the additional assumption that $\|(z-A_\eps)^{-1}\|_{L(\h_\eps)}$ be uniformly bounded.

We remark that a similar statement to part (i) in the above theorem has already been proven in \cite{MNP,Post2006}. Our result is an extension of theirs in three ways. First, they considered only \emph{sectorial} operators, which can be defined via a sesquilinear form, whereas we treat general closed operators. Second, our assumptions on the identification operators $I_\eps,\,J_\eps$ are less restrictive. Third, the spectral inclusion results (ii) and (iii) are not at all considered in \cite{MNP}.

Furthermore, convergence of spectra and pseudospectra of operators on varying spaces have been studied in \cite{B16,B18,BoegliSiegl,Hansen11} in the situation where all spaces $\h,\h_\eps$ are subspaces of a common ``large'' Hilbert space $\h_0$ and $I_\eps$ plays the role of a projection operator. In this situation, an analogue of Theorem \ref{th:mainth} has been shown in \cite{B16}. 

Note that we do not assume any connection between the spaces $\h$ and $\h_\eps$ besides that introduced by Definition \ref{def:extended_norm-resolvent}.

From Theorem \ref{th:mainth} we immediately recover two classical results about spectral convergence.
\begin{corollary}\label{cor:classical1}
		If $A_\eps,A$ are selfadjoint and bounded below for almost all $\eps>0$ and $A_\eps\to A$ in extended norm resolvent sense, then for all bounded open $V\subset\mathbb C,$ one has $\sigma(A_\eps )\cap V\rightarrow\sigma(A )\cap V$ in Hausdorff sense.
\end{corollary}
\begin{proof}
	By selfadjointness and boundedness from below of the operators involved, we have that $\rho(A)$ is connected.
	For given $r>0$, let $\delta>0$ and define the compact set $K:=\overline{B_r(0)}\setminus U_\delta(\sigma(A ))$, where $U_\delta(\cdot)$ denotes the open $\delta$-neighbourhood. By (i) we have that $K\subset \rho(A_\eps )$ for $\eps$ small enough. This shows that $\overline{B_r(0)\cap\sigma(A_\eps )}\subset \overline{B_r(0)\cap U_\delta(\sigma(A ))}$. 
	
	To see the converse inclusion $B_r(0)\cap \sigma(A )\subset B_r(0)\cap U_\delta(\sigma(A _\eps))$, let us argue by contradiction and suppose that there exists $\delta_0>0$ such that
	\begin{equation*}
		\overline{B_r(0)\cap\sigma(A )}\nsubseteq \overline{B_r(0)\cap U_{\delta_0}(\sigma(A _\eps))}\qquad\forall\eps>0.
	\end{equation*}
	By this assumption, there exists a sequence $(\lambda_\eps)$ in $\overline{B_r(0)\cap\sigma(A)}$ such that  $\dist(\lambda_\eps,\sigma(A_\eps))\ge\delta_0$ for all $\eps>0$. Since $(\lambda_\eps)$ is bounded, we can extract a subsequence $\lambda_{\eps'}\to\lambda_0\in \overline{B_r(0)}\cap\sigma(A)$.
	It follows that
	\begin{equation*}
		\overline{B_{\nicefrac{\delta_0}{2}}(\lambda_0)}\subset\rho(A _{\eps'})\qquad \text{for almost all }\eps'>0.
	\end{equation*}
	Since for all $\eps>0$ we have $\sigma(A_\eps)\subset[0,\infty)$, we can connect $B_{\nicefrac{\delta_0}{2}} $ to $\{z_0\}$ by a curve lying in $B_r(0) \cap\rho(A_\eps)$ and use Theorem \ref{th:mainth} (ii) to conclude that $\overline{B_{\nicefrac{\delta_0}{2}}(\lambda_0)}\subset\rho(A )$, which contradicts the fact that $\lambda_0\in\sigma(A )$.
	
	Since $\delta>0$ was arbitrary, the desired Hausdorff convergence follows.
\end{proof}


\begin{corollary}\label{cor:classical2}
	If $A_\eps\to A$ in extended norm resolvent sense and $(z_0-A)^{-1},\,(z_0-A_\eps)^{-1}$ are compact for all $\eps>0$ then for all bounded open $V\subset\mathbb C,$ one has $\sigma(A_\eps )\cap V\rightarrow\sigma(A )\cap V$ in Hausdorff sense.
\end{corollary}
\begin{proof}
	Compactness of the resolvent implies that for all $\eps>0$, $\rho(A_\eps),\,\rho(A)$ are connected and the spectra of $A_\eps,\,A$ consist of isolated points only. Hence the assertion follows from (i), (iii) of Theorem \ref{th:mainth}.
\end{proof}

Classical proofs of the statements in Corollaries \ref{cor:classical1} \ref{cor:classical2}, in the situation where $\h_\eps\equiv\h$ for all $\eps>0$ can be found in \cite{RS1,Kato}.

\section{Proof of Theorem \ref{th:mainth}}

In this section we will prove Theorem \ref{th:mainth}. Although the main ideas in the proof of the first part (i) are the same as in \cite{MNP}, we repeat the reasoning here to account for our differing notation and our more general hypotheses. 

\paragraph{Proof of (i).}
By assumption we have $z_0\in\rho(A_\eps)$ for all $\eps>0$ and $z_0\in\rho(A)$ and the operator norms $\bigl\|(z_0-A_\eps)^{-1}\bigr\|_{L(\h_\eps,\V_\eps)}$ are finite. Indeed, we have even more:
\begin{lemma}\label{lem:HVeps<HHeps}
	For $z\in\rho(A_\eps)$ one has
	\begin{align}\label{eq:(3.10)_Post}
		\bigl\|(z-A_\eps)^{-1}\bigr\|_{L(\h_\eps,\V_\eps)} \leq  1 + (1+|z|) \bigl\|(z-A_\eps)^{-1}\bigr\|_{L(\h_\eps)}.
	\end{align}
\end{lemma}
\begin{proof}
Let $z\in\rho(A_\eps)$. Then
\begin{align*}
	\bigl\|(z-A_\eps)^{-1}u\bigr\|_{\V_\eps} &= \|(z-A_\eps)^{-1}u\|_{\h_\eps}+\|A_\eps(z-A_\eps)^{-1}u\|_{\h_\eps}\\
	&= \|(z-A_\eps)^{-1}u\|_{\h_\eps}+\|u-z(z-A_\eps)^{-1}u\|_{\h_\eps}\\
	&\leq \|u\|_{\h_\eps} + (1+|z|)\|(z-A_\eps)^{-1}u\|_{\h_\eps}
\end{align*}
\end{proof} 

The next lemma is technical, but central to the argument. It shows that if $\eps$ is small and $\l\|\res{z,A}\r\|_{ L(\h)}$ is uniformly bounded, then $\l\|\res{z,A_\eps}\r\|_{ L(\h_\eps)}$ is uniformly bounded.

\begin{lemma}\label{lemma:R<l=>R_eps<L}
	For every $l,r>0$ there exist $\delta>0$ and $L>0$ such that if 
	\begin{enumi}
		\item $\bigl\|J_\eps(z_0-A_\eps)^{-1} -(z_0-A)^{-1}J_\eps\bigr\|_{ L(\h_\eps,\h)} < \delta, $
		\item $\l\|\res{z,A}\r\|_{ L(\h)}\leq l$,
		\item $\|\id_{\h_\eps}-I_\eps J_\eps\|_{L(\V_\eps,\h_\eps)}<\f{1}{2(|z_0|+r)}$,
		\item $z \in \rho(A_\eps)\cap\rho(A)\cap B_r(0)$
	\end{enumi}
	then $\l\|\res{z,A_\eps}\r\|_{ L(\h_\eps)}\leq L$.
\end{lemma}
The useful point in this lemma is that $L$ does not depend  on $z$ as long as $z\in\rho(A_\eps)\cap\rho(A)\cap B_r(0)$ and $\l\|\res{z,A}\r\|_{ L(\h)}\leq l$.
\begin{proof}
	We use the shorthand notation $R_\eps(z):=(z-A_\eps)^{-1}$ and $R(z):=(z-A)^{-1}$. For $z\in\rho(A_\eps)\cap\rho(A)\cap B_r(0)$ define
	\begin{align*}
		V(z) :=  J_\eps R_\eps(z) - R(z)J_\eps.
	\end{align*}
	The resolvent identity can be used to show that
	\begin{align*}
		\bigl( R(z_0) - R(z) \bigr)J_\eps R_\eps(z)R_\eps(z_0) = R(z)R(z_0)J_\eps \bigl( R_\eps(z_0) - R_\eps(z) \bigr)	\end{align*}
	which implies
	\begin{align*}
		R(z_0)V(z)R_\eps(z_0) = R(z)V(z_0)R_\eps(z)
	\end{align*}
	or
	\begin{align}\label{eq:V(z)}
		V(z) &= (z_0-A)R(z)V(z_0)R_\eps(z)(z_0-A_\eps) \\
		&= \bigl(\id_{\h} - (z-z_0)R(z)\bigr)V(z_0)\bigl(\id_{\h_\eps} - (z-z_0)R_\eps(z)\bigr)
	\end{align}
	on $\dom(A)$ and thus on $\h_\eps$ by density. Using our assumptions we deduce that
	\begin{align}\label{eq:||V(z)||}
		\l\|V(z)\r\|_{ L(\h_\eps,\h)} &\leq \delta\bigl( 1+|z-z_0|\|R_\eps(z)\|_{ L(\h_\eps)} \bigr)\bigl( 1+|z-z_0|l\bigr).
	\end{align}
	Now, decompose $R_\eps(z)$ as
	\begin{align}\label{eq:R_eps_representation}
		R_\eps(z) &= I_\eps(J_\eps R_\eps(z) - R(z)J_\eps) + I_\eps R(z)J_\eps + (\id_{\h_\eps}-I_\eps J_\eps)R_\eps(z)
	\end{align}
	This representation, together with \eqref{eq:||V(z)||} shows that
	\begin{align*}
		\|R_\eps(z)\|_{ L(\h_\eps)} &\leq \|I_\eps\|_{ L(\h,\h_\eps)} \l\|V(z)\r\|_{ L(\h_\eps,\h)} + \|I_\eps\|_{ L(\h,\h_\eps)}\|J_\eps\|_{ L(\h_\eps,\h)} \|R(z)\|_{ L(\h)}\\
		&\qquad + \|\id_{\h_\eps}-I_\eps J_\eps\|_{L(\V_\eps,\h_\eps)}\|R_\eps(z)\|_{L(\h_\eps,\V_\eps)}\\
		&\leq \delta M \bigl( 1+|z-z_0|\|R_\eps(z)\|_{ L(\h_\eps)} \bigr)\bigl( 1+|z-z_0|l\bigr) + M^2 l \\
		&\qquad + \f{1}{2(|z_0|+r)}\|R_\eps(z)\|_{L(\h_\eps,\V_\eps)}
		\end{align*}
		To estimate the last term on the right hand side we apply Lemma \ref{lem:HVeps<HHeps} to obtain
		\begin{align*}
		\|R_\eps(z)\|_{ L(\h_\eps)} &\leq \delta M( 1+|z-z_0|l)|z-z_0|\|R_\eps(z)\|_{ L(\h_\eps)} + \delta M( 1+|z-z_0|l) + M^2l \\
		&\qquad + \f{1}{2(|z_0|+r)} \big( 1+(1+|z|)\|R_\eps(z)\|_{L(\h_\eps)} \big)\\
		&\leq \|R_\eps(z)\|_{L(\h_\eps)}\left[  \delta M( 1+(|z_0|+r)l)(|z_0|+r)+ \f{1}{2(|z_0|+r)}(|z_0|+r) \right] \\
		&\qquad + \delta M( 1+(|z_0|+r)l) + M^2l + \f{1}{2(|z_0|+r)}
	\end{align*}
	Thus, if we choose $\delta<\f{1}{4M( 1+(|z_0|+r)l)(|z_0|+r)}$, we obtain the estimate
	\begin{align*}
		\|R_\eps(z)\|_{ L(\h_\eps)} &\leq \|R_\eps(z)\|_{L(\h_\eps)}\left[  \f14 + \f12 \right]  + \delta M( 1+(|z_0|+r)l) + M^2l + \f{1}{2(|z_0|+r)}
	\end{align*}
	and hence
	\begin{align*}
		\|R_\eps(z)\|_{ L(\h_\eps)} &\leq 4\left(\delta M( 1+(|z_0|+r)l) + M^2l + \f{1}{2(|z_0|+r)}\right) \\
		&=  4M^2l + \f{3}{2(|z_0|+r)} \\
		&=:L
	\end{align*}
	uniformly for $z\in\rho(A_\eps)\cap\rho(A)\cap B_r(0)$.
\end{proof}

\begin{prop}\label{pr:Generalised_spectral_convergence1}
	 Let $A_\eps:\h_\eps\supset\dom(A_\eps)\to\h_\eps$ converge to $A:\h\supset\dom(A)\to\h$ in extended norm resolvent sense. Then for every compact, connected $K\subset\rho(A)$ such that $K\cap\rho(A_\eps)\neq\emptyset$ for $\eps$ small enough there exists $\eps_0>0$ such that $K\subset\rho(A_\eps)$ for all $\eps\in(0,\eps_0)$.
\end{prop}
\begin{proof}
	We use the notation from the previous proof. Let $K\subset\rho(A)$ be compact and choose $r>0$ such that $K\subset B_r(0)$. Denote 
	\begin{align*}
		l:=\sup_{z\in K}\|R(z)\|_{L(\h_\eps)}<\infty
	\end{align*}
	and choose $\delta>0$ as in Lemma \ref{lemma:R<l=>R_eps<L} and $\eps_0>0$ such that $\bigl\|J_\eps(z_0-A_\eps)^{-1} - (z_0-A)^{-1}J_\eps\bigr\|_{L(\h_\eps,\h)}<\delta$ and $\|\id_{\h_\eps}-I_\eps J_\eps\|_{L(\V_\eps,\h_\eps)}<\f{1}{2(|z_0|+r)}$ for all $\eps\in(0,\eps_0)$, which is possible by norm resolvent convergence. Let $K_\eps:=\rho(A_\eps)\cap K$, which is non-empty by assumption and by definition relatively open in $K$.
	
	We will show that $K_\eps$ is also relatively closed in $K$ which by connectedness of $K$ implies $K_\eps=K$. To this end, let $(z_n)$ be a sequence in $K_\eps$ converging to $z\in K$. By Lemma \ref{lemma:R<l=>R_eps<L}, the sequence $\big(\|R_\eps(z_n)\|_{L(\h_\eps)}\big)_{n\in\N}$ is bounded and hence $z\in\rho(A_\eps)$. Hence, $K_\eps$ is closed in $K$ and the proof is completed.
\end{proof}

Proposition \ref{pr:Generalised_spectral_convergence1} is almost what we want. It only remains to remove the assumptions that $K$ be connected and that $K\cap\rho(A_\eps)\neq\emptyset$. This will be done in the following.

\paragraph{\textit{Conclusion of Part (i).}} Let $K\subset\rho(A)$ be compact. We decompose $K$ into its connected components $K=\bigcup_{i\in I}K_i$, where $I$ is some appropriate index set. Next, choose for each $i\in I$ a connected, open, bounded set $U_i$ such that $\overline{K_i}\subset U_i\subset\rho(A)$. Then for each $i\in I$, the set $\overline {U_i}$ is connected, compact and contained in $\rho(A)$. 

Next, choose a curve $\gamma$ in $\rho(A)$ that connects $\overline {U_i}$ with $\{z_0\}$. Then the set $K':=\overline{U_i}\cup\gamma$ is compact, connected and contained in $\rho(A)$. Moreover, since $z_0\in K'$, one has $\rho(A_\eps)\cap K'\neq\emptyset$ for all $\eps>0$ and applying Proposition \ref{pr:Generalised_spectral_convergence1} we conclude that there exists $\eps_i>0$ such that $\overline{U_i}\cup\gamma\subset \rho(A_{\eps})$ for all $\eps\in(0,\eps_i)$. Since $i\in I$ was arbitrary, such an estimate exists for every $i$. 

But since $K$ is compact and the $U_i$ form an open covering of $K$, there exists a finite subset $F\subset I$ such that $K\subset \bigcup_{i\in F}U_i$. It follows immediately that $K\subset \rho(A_{\eps})$ for all $\eps$ smaller than $\min\{\eps_i\,|\,i\in F\}$.

\paragraph{Proof of (ii).}
The proof of part (ii) of Theorem \ref{th:mainth} is similar to the previous one, but has some crucial differences that we will highlight in due course. 

We start with an analogue of Lemma \ref{lemma:R<l=>R_eps<L}.
\begin{lemma}\label{lemma:R_eps<l=>R<L}
	For every $l,r>0$, there exists $L>0$ such that if 
	\begin{enumi}
		\item $\bigl\|J_\eps(z_0-A_\eps)^{-1} -(z_0-A)^{-1}J_\eps\bigr\|_{ L(\h_\eps,\h)} \to 0$ as $\eps\to 0$,
		\item $K\subset\C$ compact with $K\subset \rho(A_\eps)\cap\rho(A)\cap B_r(0)$ for almost all $\eps>0$,
		\item $\l\|\res{z,A_\eps}\r\|_{ L(\h_\eps)}\leq l$ for all $z \in K$,
	\end{enumi}
	then one has $\l\|\res{z,A}\r\|_{ L(\h)}\leq L$ for all $z \in K$.
\end{lemma} 
\begin{proof}
	Defining $V(z)$ as in the proof of Lemma \ref{lemma:R<l=>R_eps<L}, the same reasoning leads to the estimate
	\begin{align*}
		\l\|V(z)\r\|_{L(\h_\eps,\h)} &\leq \bigl\|J_\eps R_\eps(z_0) - R(z_0)J_\eps\bigr\|_{ L(\h_\eps,\h)} \bigl( 1+|z-z_0|\|R(z)\|_{L(\h)} \bigr)\bigl( 1+|z-z_0|l\bigr).
	\end{align*}
	for $z\in K$. Now we decompose $R(z)$ as
	\begin{align*}
		R(z) = (R(z)J_\eps - J_\eps R_\eps(z))I_\eps + J_\eps R_\eps(z) I_\eps + R(z)(\id_{\h}-J_\eps I_\eps)
	\end{align*}
	For $u\in\h$ with $\|u\|_{\h}=1$ we immediately obtain
	\begin{align*}
		\|R(z)u\|_{L(\h)} &\leq \|I_\eps\|_{L(\h,\h_\eps)}\|V(z)\|_{L(\h_\eps,\h)} + \|R(z)(u-J_\eps I_\eps u)\|_\h \\
		&\qquad + \|I_\eps\|_{L(\h,\h_\eps)}\|J_\eps\|_{L(\h_\eps,\h)}\|R_\eps(z)\|_{L(\h_\eps)}\\
		&\leq M \bigl\|J_\eps R_\eps(z_0) - R(z_0)J_\eps\bigr\|_{ L(\h_\eps,\h)} \bigl( 1+|z-z_0|\|R(z)\|_{ L(\h)} \bigr)\bigl( 1+|z-z_0|l\bigr)\\
		&\qquad + M^2l + \|R(z)\|_{L(\h)} \|(u-J_\eps I_\eps u)\|_\h
	\end{align*}
	Next, choose $\eps=\eps(u)$ so small that $\|u-J_{\eps(u)} I_{\eps(u)} u\|_\h<\f12$. Without loss of generality we can assume that $\bigl\|J_{\eps(u)} R_{\eps(u)}(z_0) - R(z_0)J_{\eps(u)}\bigr\|_{ L(\h_{\eps(u)},\h)}$ is smaller than $\delta:=\f{1}{4M( 1+(|z_0|+r)l)(|z_0|+r)}$. We obtain
	\begin{align*}
		\f12 \|R(z)u\|_{L(\h)} &\leq M \delta \bigl( 1+|z-z_0|\|R(z)\|_{ L(\h)} \bigr)\bigl( 1+|z-z_0|l\bigr) + M^2l \\
		&\leq M \delta \bigl( 1+(|z_0|+r)\|R(z)\|_{ L(\h)} \bigr)\bigl( 1+(|z_0|+r)l\bigr) + M^2l\\
		&= \delta M \bigl( 1+(|z_0|+r)l\bigr)(|z_0|+r)\|R(z)\|_{ L(\h)} + M\delta\bigl( 1+(|z_0|+r)l\bigr) + M^2l\\
		&= \f14\|R(z)\|_{ L(\h)} + \f{1}{4(|z_0|+r)} + M^2l,
	\end{align*}
	where the right hand side does not depend on $u$. Taking the supremum over all $u$ with $\|u\|_\h=1$ we conclude that
	\begin{align*}
		\|R(z)\|_{L(\h)} &\leq \f{1}{4(|z_0|+r)} + M^2l\\
		&=:L
	\end{align*}
	uniformly for $z\in K$.

\end{proof}

\begin{prop}\label{pr:Generalised_spectral_convergence2}
	Let $A_\eps\to A$ in extended norm resolvent sense, and assume that $K\subset\C$ is connected, compact such that $K\subset\rho(A_\eps)$, $\|(z-A_\eps)^{-1}\|_{L(\h_\eps)}$ uniformly bounded in $z\in K$ and $\eps>0$ and $K\cap\rho(A)\neq\emptyset$ and there exists $r>0$ such that $K\subset B_r(0)$ for almost all $\eps>0$. Then one has $K\subset\rho(A)$.
\end{prop}
\begin{proof}
	Let $K\subset\C$ be compact, connected such that $K\subset\rho(A_\eps)$ and $K\cap\rho(A)\neq\emptyset$ for almost all $\eps>0$. Choose $l>0$ such that $\|(z-A_\eps)^{-1}\|_{L(\h_\eps)}\leq l$ for all $z\in K,\,\eps>0$. Choose $r>0$ such that $K\subset B_r(0)$ for almost all $\eps>0$. 
	
	By assumption, we have $K':=K\cap\rho(A)\neq\emptyset$ and $K'$ is relatively open in $K$. We will show that $K'$ is also relatively closed in $K$. Let $(z_n)\subset K'$ be a sequence such that $z_n\to z\in K$. Then by Lemma \ref{lemma:R_eps<l=>R<L} the sequence $\|(z_n-A)^{-1}\|_{L(\h)}$ is uniformly bounded and hence $z\in\rho(A)$. Indeed, since $K'\subset\rho(A)\cap\rho(A_\eps)\cap B_r(0)$, it satisfies the assumptions of Lemma \ref{lemma:R_eps<l=>R<L}.
	
	Since $K'\neq\emptyset$ and $K$ is connected, we conclude that $K'=K$.

%
\end{proof}

\paragraph{\textit{Conclusion of Part (ii).}}

Let $K\subset\C$ be compact and connected such that $K\subset\rho(A_\eps)$ for all $\eps>0$. By assumption we may choose a curve $\gamma$ such that $z_0\in K\cup\gamma$. But now $K\cup\gamma$ satisfies the assumptions of Proposition \ref{pr:Generalised_spectral_convergence2} and hence $K\subset K\cup\gamma\subset\rho(A)$.

\paragraph{Proof of (iii).}

Let $\lambda\in\sigma(A)$ be an isolated point and let $\delta>0$ small enough that $B_\delta(\lambda)\setminus\{\lambda\}\subset\rho(A)$. Define $K:=\overline{B_\delta(\lambda)}\setminus B_{\frac{\delta}{2}}(\lambda)$. Then by (i) of Theorem \ref{th:mainth} we know that $K\subset\rho(A_\eps)$ for $\eps$ small enough and there exists $L>0$ such that $\|(z-A_\eps)^{-1}\|_{L(\h_\eps)}\leq L$ for all $z\in K$ and $\eps>0$ (cf. Lemma \ref{lemma:R<l=>R_eps<L}). 

Next, define $K':=\overline{B_\delta(\lambda)}$. Then
\begin{itemize}
	\item either there exists $\eps_0>0$ such that $K'\cap\sigma(A_\eps)\neq\emptyset$ for all $\eps\in(0,\eps_0)$, or 
	\item there exists a sequence $\eps_n\searrow 0$ such that $K'\subset\rho(A_{\eps_n})$ for all $n$.
\end{itemize}
In the first case, we conclude that there is a spectral point of $A_\eps$ in  $\overline{B_\delta(\lambda)}$. In the second case, we argue as follows. 

By construction, $\|R_\eps(z)\|_{L(\h_\eps)}\leq L$ on $K'\setminus B_{\f{\delta}{2}}(\lambda)$. We know that $\|R_\eps(z)\|_{L(\h_\eps)}$ cannot be uniformly bounded on all of $K'$, since Theorem \ref{th:mainth} (ii) would imply $K'\subset\rho(A)$, which is false, since $\lambda\in K'$. Hence we must have that $\|R_\eps\|_{L(\h_\eps)}$ is unbounded on $B_{\f{\delta}{2}}(\lambda)$, i.e. there exists a sequence $(z_n)\subset B_{\f{\delta}{2}}(\lambda)$ such that $\|R_{\eps_n}(z_n)\|_{L(\h_\eps)}\to\infty$ as $n\to\infty$.

If for infinitely many $n$ there is no spectral point of $A_{\eps_n}$ in $B_{\f{\delta}{2}}(\lambda)$, then we conclude by the maximum principle for subharmonic functions that there exists another sequence of points $\tilde z_n$ on the boundary of $B_{\f{\delta}{2}}(\lambda)$ such that $\|R_{\eps_n}(\tilde z_n)\|_{L(\h_{\eps_n})}\to\infty$. But the boundary of $B_{\f{\delta}{2}}(\lambda)$ is included in $K$ on which we have $\|R_\eps\|_{L(\h_\eps)}\leq L$ for all $\eps>0$ - a contradiction. Hence there must be a spectral point of $A_\eps$ in $B_{\f{\delta}{2}}(\lambda)$ for $\eps$ small enough. 

We have shown that in either case, we necessarily have $B_{\delta}(\lambda)\cap\sigma(A_\eps)\neq\emptyset$ for $\eps$ small enough. Since $\delta>0$ was arbitrary, the claim follows.

\qed

\section{Concluding Remarks}

We conclude with a few remarks on the hypotheses in Definition \ref{def:extended_norm-resolvent}. It has been shown in \cite{B16} that spectral inclusion in fact holds under the milder assumption of \emph{strong} resolvent convergence. 

We will now show that an analogous statement is also true in the present situation. 
\begin{prop}\label{prop:ConcProp1}
	Assume that there exists $z_0\in\rho(A)$ such that $z_0\in\rho(A_\eps)$ for almost all $\eps>0$ and for all $u\in\h$
	\begin{align*}
		\left\| \left(I_\eps(z_0-A)^{-1} - (z_0-A_\eps)^{-1}I_\eps\right)u \right\|_{L(\h,\h_\eps)}\to 0.
	\end{align*}
	Then conclusion (ii) of Theorem \ref{th:mainth} holds.
\end{prop}
The proof of Proposition \ref{prop:ConcProp1} merely requires a version of Lemma \ref{lemma:R_eps<l=>R<L}:
\begin{lemma}\label{lemma:R_eps<l=>R<LSTRONG}
	For every $l,r>0$, there exists $L>0$ such that if 
	\begin{enumi}
		\item For all $u\in\h$ one has $\left\| \left(I_\eps(z_0-A)^{-1} - (z_0-A_\eps)^{-1}I_\eps\right)u \right\|_{L(\h,\h_\eps)}\to 0$ as $\eps\to 0$,
		\item $K\subset\C$ compact with $K\subset \rho(A_\eps)\cap\rho(A)\cap B_r(0)$ for almost all $\eps>0$,
		\item $\l\|\res{z,A_\eps}\r\|_{ L(\h_\eps)}\leq l$ for all $z \in K$,
	\end{enumi}
	then one has $\l\|\res{z,A}\r\|_{ L(\h)}\leq L$ for all $z \in K$.
\end{lemma} 
\begin{proof}
	For $z\in\rho(A_\eps)\cap\rho(A)\cap B_r(0)$, define
	\begin{align*}
		V_\eps(z):=I_\eps(z_0-A)^{-1} - (z_0-A_\eps)^{-1}I_\eps.
	\end{align*}
	An analogous calculation to the one above eq. \eqref{eq:V(z)} leads to the identity
	\begin{align*}
		V_\eps(z) = (\id_{\h_\eps} - (z-z_0)R_\eps(z))V_\eps(z_0)(\id_\h - (z-z_0)R(z)).
	\end{align*}
	This implies that for any $u\in\h$ one has the inequality
	\begin{align*}
		\|V_\eps(z)u\|_{L(\h,\h_\eps)}\leq (1+|z-z_0|l)\left\| V_\eps(z_0)(u-(z-z_0)R(z)u) \right\|_{L(\h,\h_\eps)}.
	\end{align*}
	Decomposing $R(z)$ as
	\begin{align*}
		R(z) = J_\eps(I_\eps R(z)-R_\eps(z)I_\eps) + J_\eps R_\eps I_\eps + (\id_\h - J_\eps I_\eps)R(z)
	\end{align*}
	we find that for all $u\in\h$ with $\|u\|_\h=1$ and all $\eps>0$
	\begin{align*}
		\|R(z)u\|_{L(\h)} &\leq M\|V_\eps(z)u\|_{\h} + M^2\|R_\eps(z)\|_{L(\h_\eps)} + \|(\id_\h - J_\eps I_\eps)R(z)u\|_{L(\h)}\\
		&\leq M(1+|z-z_0|l)\left\| V_\eps(z_0)(u-(z-z_0)R(z)u) \right\|_{L(\h,\h_\eps)} + M^2l \\
		&\qquad\qquad\quad +\|(\id_\h - J_\eps I_\eps)R(z)u\|_{L(\h)}.
	\end{align*}
	We immediately conclude that
	\begin{align*}
		\|R(z)u\|_{L(\h)} &\leq \limsup_{\eps\to 0} \Big(M(1+|z-z_0|l)\left\| V_\eps(z_0)(u-(z-z_0)R(z)u) \right\|_{L(\h,\h_\eps)}  \\
		&\qquad\qquad\quad + M^2l + \|(\id_\h - J_\eps I_\eps)R(z)u\|_{L(\h)}\Big)\\
		&\leq M^2l,
	\end{align*}
	by the strong convergences $V_\eps(z_0)\to 0$ and $\id_\h - J_\eps I_\eps\to 0$. Hence, $\|R(z)u\|_{L(\h)}$ is uniformly bounded for $z\in\rho(A_\eps)\cap\rho(A)\cap B_r(0)$ and $u\in\h$ with $\|u\|_\h=1$, which implies the assertion.
\end{proof}
The rest of the proof of Proposition \ref{prop:ConcProp1} follows that of Theorem \ref{th:mainth} (ii) verbatim.

 
\end{document}